\documentclass[11pt]{article}
\usepackage[a4paper,margin=25mm]{geometry}
\usepackage{amsmath,amssymb,amsthm,mathtools,bm}
\usepackage{hyperref}

\newcommand{\Z}{\mathbb{Z}}
\newcommand{\AGL}{\mathrm{AGL}_1}
\newcommand{\comm}[2]{[#1,#2]}

\newtheoremstyle{uprightthm}
  {}{}{\normalfont\upshape}{}{\bfseries}{.}{ }{}
\theoremstyle{uprightthm}
\newtheorem{theorem}{Theorem}[section]
\newtheorem{lemma}[theorem]{Lemma}
\newtheorem{definition}[theorem]{Definition}
\newtheorem{proposition}[theorem]{Proposition}
\newtheorem{corollary}[theorem]{Corollary}
\newtheorem{example}[theorem]{Example}
\newtheorem{remark}[theorem]{Remark}

\title{Cross-Commuting Nonabelian Squares in Affine Groups over Finite Commutative Principal Ideal Rings}
\author{Kenta Kasai\\
Institute of Science Tokyo\\
\texttt{kenta@ict.eng.isct.ac.jp}}
\date{}

\begin{document}
\maketitle

\begin{abstract}
We study a commutation pattern in which two affine families commute completely
across the two families while each family retains internal noncommutativity.
For one-dimensional affine groups over finite commutative rings, we prove a
local--product dichotomy. Over a finite commutative local principal ideal ring,
the common centralizer of two noncommuting affine permutations is always
abelian, so the pattern is impossible. Over a direct product of two
commutative rings whose affine groups each contain a noncommuting pair, the
same pattern is constructed by separating the two noncommuting families into
different factors. More generally, over a finite commutative principal ideal
ring, the pattern exists if and only if at least two local factors are not
isomorphic to $\mathbb F_2$. Applied to residue rings, this yields an exact
classification: $\AGL(\Z/n\Z)$ contains the pattern if and only if at least
two prime-power factors of $n$ exceed $2$. We also compare this phenomenon with
the permutation-group setting, where the same pattern is easy to realize.
\end{abstract}

\section{Introduction}
Affine groups over cyclic objects have long appeared in several neighboring
lines of algebra, including the study of holomorphs of cyclic groups,
subgroups of affine groups over finite rings, and representations and
automorphisms of those groups. The common theme in that literature is to use
the interplay between additive structure and multiplicative units to understand
how much group-theoretic complexity can occur in a seemingly simple affine
setting.

The present note isolates one concrete commutation problem inside that broader
history. We ask whether two affine families can commute completely across the
two families while each family still retains internal noncommutativity. This is
the smallest pattern in which complete compatibility across the split coexists
with genuine nonabelian behavior on both sides.

Our main point is that the answer is governed not by size alone but by ring
structure. On the local side, we prove that over a finite commutative local
principal ideal ring, the common centralizer of two noncommuting affine
permutations is always abelian. Hence the target pattern is impossible there.
On the product side, we prove that if a direct product ring has two factors
whose affine groups each admit a noncommuting pair, then one can realize the
same pattern by placing the two noncommuting families in different factors. The
natural organizing principle is therefore a local--product dichotomy.
Because finite commutative principal ideal rings decompose into local factors,
this dichotomy in fact yields a complete classification in that setting.

The closest surrounding literature studies holomorphs of cyclic groups and
subgroups of affine groups over finite rings from broader structural viewpoints.
For example, Spaggiari studies mutually normalizing regular subgroups in the
holomorph of a cyclic group of prime-power order~\cite{Spaggiari2023}, Sato
studies automorphisms of holomorphs of cyclic groups~\cite{Sato2024Holomorph},
Pradhan and Sury study representation-theoretic aspects of holomorphs of cyclic
$p$-groups~\cite{PradhanSury2022}, Bayramgulov and Pogorelov study subgroups of
affine groups over Galois rings~\cite{BayramgulovPogorelov2024}, and Caranti,
Dalla Volta, and Sala study abelian regular subgroups of affine groups
~\cite{CarantiDallaVoltaSala2006}.

On the product side, our existence statement is essentially compatible with the
classical observation of Mills that if a group splits as a direct product of
characteristic subgroups, then its holomorph splits as the direct product of
the holomorphs of those factors~\cite{Mills1961}. Thus, once a ring or cyclic
object decomposes into coprime components that each support a noncommuting
affine pair, separating those pairs into different factors is a natural
direct-product construction. By contrast, we have not found a prior source
that explicitly isolates the local
common-centralizer obstruction proved here, or the resulting nonexistence of
the cross-commuting nonabelian square over finite commutative local principal
ideal rings. To the best of our knowledge, the safest summary is therefore
that the product-side existence is essentially implicit in classical holomorph
decomposition/direct-product arguments, whereas the local
obstruction/common-centralizer theorem appears to be new.

The residue-ring case of arithmetic modulo a positive integer then becomes a
concrete specialization of this general picture. Prime-power moduli belong to
the local side, whereas many composite moduli fall on the product side through
the Chinese remainder theorem (CRT). Appendix~A records the coding-theoretic
motivation that originally led to this question, but the body of the paper is
purely group-theoretic and ring-theoretic.

\section{Affine permutations over commutative rings}
Let $R$ be a commutative ring with identity. An affine permutation on $R$ is a
map
\[
x\longmapsto ax+b,
\]
where $a\in R^\times$ and $b\in R$. We identify this map with the pair $(a,b)$
and write $\AGL(R)$ for the corresponding affine group. The composition law is
\[
(a,b)\circ(c,d)=(ac,ad+b).
\]

The following criterion is the basic algebraic input.

\begin{lemma}[Commutation criterion]
\label{lem:comm-criterion-en}
Two affine permutations $(a,b)$ and $(c,d)$ in $\AGL(R)$ commute if and only if
\[
(a-1)d=(c-1)b.
\]
\end{lemma}

\begin{proof}
Since
\[
(a,b)\circ(c,d)=(ac,ad+b),\qquad
(c,d)\circ(a,b)=(ca,cb+d),
\]
and $R$ is commutative, the linear parts agree automatically. The two affine
maps commute exactly when the translation parts agree, namely when
$ad+b=cb+d$, and this is equivalent to $(a-1)d=(c-1)b$.
\end{proof}

\begin{definition}[Cross-commuting nonabelian square]
\label{def:square-en}
A quadruple $(F_0,F_1,G_0,G_1)$ in $\AGL(R)$ is called a
\emph{cross-commuting nonabelian square} if
\[
\comm{F_i}{G_j}=1\qquad (i,j\in\{0,1\}),
\]
but
\[
\comm{F_0}{F_1}\ne 1,\qquad \comm{G_0}{G_1}\ne 1.
\]
\end{definition}

This definition isolates the smallest pattern in which the two families are
completely compatible across the split while each family still carries genuine
noncommutativity of its own.

\section{The local obstruction}
From now on, let $R$ be a finite commutative local principal ideal ring. Then
$R$ has a unique maximal ideal $\mathfrak m=(\pi)$, every ideal is generated by
a power of $\pi$, and there is a smallest positive integer $\ell$ such that
$\pi^\ell=0$; see, for example,~\cite{McDonald1974}. For a nonzero element
$x\in R$, we write $v_\pi(x)=t$ if $x\in (\pi^t)\setminus (\pi^{t+1})$.

The proof below uses the Smith normal form over the principal ideal ring $R$.
In the present $2\times 2$ situation, this amounts to diagonalizing the matrix
of simultaneous commutation equations by invertible row and column operations.
It is a standard reduction in the structure theory of modules over principal
ideal rings; see, for example,~\cite{McDonald1974}.

\begin{theorem}[Local obstruction]
\label{thm:local-obstruction-en}
Let $R$ be a finite commutative local principal ideal ring. If
$F_0,F_1\in\AGL(R)$ do not commute, then the common centralizer
\[
C(F_0)\cap C(F_1)
\]
is abelian. In particular, no cross-commuting nonabelian square exists in
$\AGL(R)$.
\end{theorem}

\begin{proof}
Write
\[
F_0=(a_0,b_0),\qquad F_1=(a_1,b_1),
\]
and set
\[
s_0:=a_0-1,\qquad s_1:=a_1-1.
\]
By Lemma~\ref{lem:comm-criterion-en}, the noncommutativity of $F_0$ and $F_1$
means that
\[
\Delta:=s_0b_1-s_1b_0
\]
is nonzero.

Now let $G=(c,d)$ commute with both $F_0$ and $F_1$, and write $u:=c-1$. By
Lemma~\ref{lem:comm-criterion-en}, the pair $(u,d)$ satisfies
\[
b_0u-s_0d=0,\qquad b_1u-s_1d=0.
\]
Thus $(u,d)^{\mathsf T}$ lies in the kernel of the matrix
\[
M:=
\begin{pmatrix}
b_0 & -s_0\\
b_1 & -s_1
\end{pmatrix}.
\]
Its determinant is exactly $\Delta$.

Since $R$ is a principal ideal ring, $M$ admits a Smith normal form:
there exist invertible matrices $U,V\in \operatorname{GL}_2(R)$ and integers
$0\le \alpha\le \beta\le \ell$ such that
\[
UMV=\operatorname{diag}(\pi^\alpha,\pi^\beta),
\qquad
\alpha+\beta=v_\pi(\Delta)<\ell.
\]
Hence
\[
\ker(M)=V\bigl(\pi^{\ell-\alpha}R\times \pi^{\ell-\beta}R\bigr).
\]

Take two elements $G=(1+u,d)$ and $G'=(1+u',d')$ in $C(F_0)\cap C(F_1)$. Their
parameter vectors lie in $\ker(M)$, so there exist $r,s,r',s'\in R$ with
\[
\binom{u}{d}
=
V\binom{\pi^{\ell-\alpha}r}{\pi^{\ell-\beta}s},
\qquad
\binom{u'}{d'}
=
V\binom{\pi^{\ell-\alpha}r'}{\pi^{\ell-\beta}s'}.
\]
Therefore
\[
ud'-u'd
=
\det(V)\,\pi^{2\ell-\alpha-\beta}(rs'-r's).
\]
Because $\alpha+\beta<\ell$, we have $2\ell-\alpha-\beta>\ell$, and therefore
$\pi^{2\ell-\alpha-\beta}=0$. Hence $ud'=u'd$. By
Lemma~\ref{lem:comm-criterion-en}, $G$ and $G'$ commute.

Thus every two elements of $C(F_0)\cap C(F_1)$ commute, so this common
centralizer is abelian. The final assertion follows immediately from
Definition~\ref{def:square-en}.
\end{proof}

The theorem says that in the local case, once one family contains genuine
noncommutativity, the entire opposite side is forced into an abelian common
centralizer. The next corollaries simply record this in the square language and
then in the family language.

\begin{corollary}[Family version]
\label{cor:family-local-en}
Let $R$ be a finite commutative local principal ideal ring. Let
$\mathcal F=\{F_i\}_{i\in I}$ and $\mathcal G=\{G_j\}_{j\in J}$ be families in
$\AGL(R)$ such that every $F_i$ commutes with every $G_j$. If the family
$\mathcal F$ contains a noncommuting pair, then the family $\mathcal G$ is
pairwise commuting. By symmetry, if $\mathcal G$ contains a noncommuting pair,
then $\mathcal F$ is pairwise commuting. In particular, two cross-commuting
families cannot both be nonabelian.
\end{corollary}

\begin{proof}
Assume that $F_r$ and $F_s$ do not commute for some $r,s\in I$. Since every
$G_j$ commutes with every element of $\mathcal F$, each $G_j$ belongs to
$C(F_r)\cap C(F_s)$. By Theorem~\ref{thm:local-obstruction-en}, this common
centralizer is abelian. Hence any two elements of $\mathcal G$ commute. The
second assertion follows by interchanging the roles of $\mathcal F$ and
$\mathcal G$.
\end{proof}

\section{Product constructions}
The obstruction above disappears once the ring splits into factors whose affine
groups already contain noncommuting pairs. This is the algebraic counterpart of
the CRT phenomenon for residue rings, and it is also compatible with Mills's
classical holomorph decomposition under direct products of characteristic
factors~\cite{Mills1961}.

\begin{theorem}[Product construction]
\label{thm:product-construction-en}
Let $R_1$ and $R_2$ be commutative rings with identity. Assume that
$\AGL(R_1)$ contains a noncommuting pair and that $\AGL(R_2)$ contains a
noncommuting pair. Then $\AGL(R_1\times R_2)$ contains a cross-commuting
nonabelian square.
\end{theorem}

\begin{proof}
Choose noncommuting affine permutations $A_0,A_1\in\AGL(R_1)$ and
$B_0,B_1\in\AGL(R_2)$. Define affine permutations on $R_1\times R_2$ by
\[
F_i(x_1,x_2)=(A_i(x_1),x_2),\qquad
G_j(x_1,x_2)=(x_1,B_j(x_2)).
\]
Every $F_i$ acts only on the first factor and every $G_j$ acts only on the
second factor, so all cross pairs commute. On the other hand,
\[
F_0F_1=F_1F_0 \iff A_0A_1=A_1A_0,
\]
and similarly
\[
G_0G_1=G_1G_0 \iff B_0B_1=B_1B_0.
\]
Thus $F_0,F_1$ do not commute and $G_0,G_1$ do not commute.
\end{proof}

The product construction should be read as the exact opposite of the local
obstruction: once the ring splits into independent factors that each already
carry a noncommuting pair, the two noncommuting families can be separated into
different components, so the cross commutation becomes automatic.

\begin{proposition}[Product family construction]
\label{prop:product-family-en}
Let $R_1$ and $R_2$ be commutative rings with identity. Let
$\mathcal A\subset \AGL(R_1)$ and $\mathcal B\subset \AGL(R_2)$ be arbitrary
families. Then there exist families
$\widetilde{\mathcal A},\widetilde{\mathcal B}\subset \AGL(R_1\times R_2)$
such that every element of $\widetilde{\mathcal A}$ commutes with every
element of $\widetilde{\mathcal B}$, while the internal commutation relations
inside $\widetilde{\mathcal A}$ and inside $\widetilde{\mathcal B}$ are
exactly the same as those inside $\mathcal A$ and $\mathcal B$, respectively.
In particular, if both $\mathcal A$ and $\mathcal B$ are nonabelian, then
$\AGL(R_1\times R_2)$ contains two cross-commuting nonabelian families.
\end{proposition}

\begin{proof}
For each $A\in\mathcal A$, define
\[
\widetilde A(x_1,x_2)=(A(x_1),x_2),
\]
and for each $B\in\mathcal B$, define
\[
\widetilde B(x_1,x_2)=(x_1,B(x_2)).
\]
Every $\widetilde A$ acts trivially on the second factor and every
$\widetilde B$ acts trivially on the first, so every cross pair commutes.
Moreover,
\[
\widetilde A_1\widetilde A_2=\widetilde A_2\widetilde A_1
\iff
A_1A_2=A_2A_1,
\]
and similarly for the $B$ side. Hence the internal commutation relations are
preserved exactly.
\end{proof}

\begin{lemma}[When a local factor supports noncommutation]
\label{lem:local-noncommuting-pair-en}
Let $S$ be a finite commutative local ring with identity. Then $\AGL(S)$
contains a noncommuting pair if and only if $S\not\cong \mathbb F_2$.
\end{lemma}

\begin{proof}
If $S\cong\mathbb F_2$, then $\AGL(S)=\{(1,0),(1,1)\}$ is abelian.

Conversely, let $\mathfrak m$ be the maximal ideal of $S$. If the residue field
$S/\mathfrak m$ has more than two elements, choose a unit $a\in S^\times$ whose
image in $S/\mathfrak m$ is not $1$. Then $a-1$ is a unit, so by
Lemma~\ref{lem:comm-criterion-en} the elements $(1,1)$ and $(a,0)$ do not
commute.

If instead $S/\mathfrak m\cong \mathbb F_2$ but $S\not\cong\mathbb F_2$, then
$\mathfrak m\ne 0$. Choose $0\ne x\in\mathfrak m$. Since $x$ lies in the
maximal ideal, $1+x$ is a unit. Again by Lemma~\ref{lem:comm-criterion-en}, the
elements $(1,1)$ and $(1+x,0)$ do not commute, because $(1+x)-1=x\ne 0$.
\end{proof}

\begin{theorem}[Classification over finite commutative principal ideal rings]
\label{thm:pir-classification-en}
Let $R$ be a finite commutative principal ideal ring, and write
\[
R\cong R_1\times \cdots \times R_t
\]
as its canonical decomposition into finite commutative local principal ideal
rings. The following are equivalent.
\begin{enumerate}
\item $\AGL(R)$ contains a cross-commuting nonabelian square.
\item There exist distinct indices $i\ne j$ such that $\AGL(R_i)$ and
$\AGL(R_j)$ each contain a noncommuting pair.
\item At least two of the local factors $R_k$ are not isomorphic to
$\mathbb F_2$.
\end{enumerate}
\end{theorem}

\begin{proof}
Identify $\AGL(R)$ with $\AGL(R_1)\times\cdots\times \AGL(R_t)$, and let
$\operatorname{pr}_k$ denote the projection to the $k$-th factor.

Assume first that $(F_0,F_1,G_0,G_1)$ is a cross-commuting nonabelian square in
$\AGL(R)$. Since $F_0$ and $F_1$ do not commute in the direct product, there
exists an index $i$ such that $\operatorname{pr}_i(F_0)$ and
$\operatorname{pr}_i(F_1)$ do not commute in $\AGL(R_i)$. Because every cross
pair commutes in $\AGL(R)$, the projected quadruple is cross-commuting in
$\AGL(R_i)$. Theorem~\ref{thm:local-obstruction-en} therefore implies that
$\operatorname{pr}_i(G_0)$ and $\operatorname{pr}_i(G_1)$ commute. By symmetry,
there exists an index $j$ such that $\operatorname{pr}_j(G_0)$ and
$\operatorname{pr}_j(G_1)$ do not commute, and then
$\operatorname{pr}_j(F_0)$ and $\operatorname{pr}_j(F_1)$ commute. Hence
$i\ne j$. This proves $(1)\Rightarrow(2)$.

Next assume $(2)$. Choose noncommuting pairs
$A_0,A_1\in \AGL(R_i)$ and $B_0,B_1\in \AGL(R_j)$ with $i\ne j$. Define
elements of $\AGL(R)$ by letting $F_0,F_1$ act as $A_0,A_1$ on the $i$-th
factor and trivially on all other factors, and letting $G_0,G_1$ act as
$B_0,B_1$ on the $j$-th factor and trivially on all other factors. Then every
$F_r$ commutes with every $G_s$, while $F_0,F_1$ and $G_0,G_1$ remain
noncommuting on their respective factors. Thus $(1)$ holds. This proves
$(2)\Rightarrow(1)$.

Finally, $(2)$ and $(3)$ are equivalent by
Lemma~\ref{lem:local-noncommuting-pair-en}, applied to each local factor
$R_k$.
\end{proof}

\section{Consequences for residue rings}
We now specialize the local--product picture to the ring $\Z/n\Z$.

\begin{corollary}[Exact classification for residue rings]
\label{cor:zn-consequences-en}
Let
\[
n=\prod_{k=1}^t p_k^{e_k}
\]
be the prime-power factorization of $n\ge 2$. The following are equivalent.
\begin{enumerate}
\item $\AGL(\Z/n\Z)$ contains a cross-commuting nonabelian square.
\item At least two prime-power factors $p_k^{e_k}$ are greater than $2$.
\item $n$ is neither a prime power nor twice an odd prime power.
\end{enumerate}
\end{corollary}

\begin{proof}
By the CRT,
\[
\Z/n\Z \cong \prod_{k=1}^t \Z/p_k^{e_k}\Z.
\]
Each factor $\Z/p_k^{e_k}\Z$ is a finite commutative local principal ideal
ring, and it is isomorphic to $\mathbb F_2$ exactly when $p_k^{e_k}=2$.
Therefore $(1)$ and $(2)$ are equivalent by
Theorem~\ref{thm:pir-classification-en}. The equivalence of $(2)$ and $(3)$ is
immediate: condition $(2)$ fails exactly when there is at most one prime-power
factor greater than $2$, namely when $n$ is a prime power or $n=2q$ with $q$ an
odd prime power.
\end{proof}

\begin{example}[The smallest residue-ring example]
The smallest modulus covered by
Corollary~\ref{cor:zn-consequences-en} is $n=12=3\cdot 4$. Taking the CRT
idempotents
\[
e_3=4,\qquad e_4=9
\]
in $\Z/12\Z$, one obtains
\[
F_0(x)=x+4,\qquad F_1(x)=5x,
\]
\[
G_0(x)=x+9,\qquad G_1(x)=7x.
\]
Then every $F_i$ commutes with every $G_j$, while $F_0$ and $F_1$ do not
commute and $G_0$ and $G_1$ do not commute.
\end{example}

\begin{example}[The modulus $n=768$]
The modulus used in~\cite{Kasai2026} is $n=768=3\cdot 256$, so
Corollary~\ref{cor:zn-consequences-en} applies. Writing the CRT idempotents as
\[
e_3=256,\qquad e_{256}=513
\]
in $\Z/768\Z$, one convenient choice is
\[
F_0(x)=x+256,\qquad F_1(x)=257x,
\]
\[
G_0(x)=x+513,\qquad G_1(x)=511x.
\]
Again, every $F_i$ commutes with every $G_j$, while
$\comm{F_0}{F_1}\ne 1$ and $\comm{G_0}{G_1}\ne 1$.
\end{example}

\section{Comparison with permutation groups}
The local obstruction above is specific to affine groups. It does not extend to
the permutation-group setting, where the same pattern is easy to realize.

\begin{proposition}[Permutation-group comparison]
\label{prop:gpm-comparison-en}
For every integer $n\ge 6$, the symmetric group $S_n$ contains a
cross-commuting nonabelian square.
\end{proposition}

\begin{proof}
Let
\[
F_0=(1\,2\,3),\qquad F_1=(1\,2),\qquad
G_0=(4\,5\,6),\qquad G_1=(4\,5)
\]
in $S_n$, with all remaining points fixed. Since $F_0$ and $F_1$ act on
$\{1,2,3\}$, they do not commute. Likewise $G_0$ and $G_1$ do not commute on
$\{4,5,6\}$. On the other hand, every $F_i$ has support disjoint from every
$G_j$, so all cross pairs commute. Thus these four permutations form a
cross-commuting nonabelian square.
\end{proof}

\begin{remark}
In particular, the existence direction is much easier for permutation groups
than for affine groups: it holds for every $n\ge 6$, including prime powers
such as $8$, $9$, and $16$. Therefore the obstruction in
Theorem~\ref{thm:local-obstruction-en} is a rigid feature of affine groups over
local rings rather than of permutation-like groups in general.
\end{remark}

\section{Conclusion}
The commutation pattern considered here is controlled by a local--product
distinction. Over a finite commutative local principal ideal ring, two
noncommuting affine permutations have an abelian common centralizer, and this
rules out the target pattern. Over a direct product ring whose two factors each
admit a noncommuting affine pair, the same pattern can be realized by
separating the two noncommuting families into different factors. From this
viewpoint, the product-side existence is essentially a direct-product phenomenon
already implicit in classical holomorph decompositions, whereas the new
ingredient is the local common-centralizer obstruction. Because finite
commutative principal ideal rings decompose into local factors, this yields a
complete global classification in that setting. The residue-ring case
$\Z/n\Z$ is therefore best viewed as a concrete specialization of a more
general ring-theoretic phenomenon.

\appendix
\section{Coding-theoretic motivation}
The present note is logically independent of coding theory, but one motivation
for the cross-commuting nonabelian square comes from permutation-based quantum
code constructions of Hagiwara--Imai type
~\cite{HagiwaraImai2007,KasaiEtAl2011,Kasai2026}. In that setting, cross
commutation is used to enforce a Calderbank--Shor--Steane (CSS) orthogonality
condition, whereas internal
noncommutation is used to avoid overly commuting configurations that create
Fossorier-type short cycles~\cite{Fossorier2004}. The following proposition
explains this point for the standard two-row block arrays.

\begin{proposition}[Why internal noncommutation is needed in the
Hagiwara--Imai setting]
\label{prop:HI-necessity-en}
Let $m=L/2\ge 3$, and consider the standard two-row Hagiwara--Imai block arrays
\[
\hat H_X=
\begin{bmatrix}
F_0 & F_1 & \cdots & F_{m-1} & G_0 & G_1 & \cdots & G_{m-1}\\
F_{m-1} & F_0 & \cdots & F_{m-2} & G_{m-1} & G_0 & \cdots & G_{m-2}
\end{bmatrix},
\]
\[
\hat H_Z=
\begin{bmatrix}
G_0^\top & G_{m-1}^\top & \cdots & G_1^\top & F_0^\top & F_1^\top & \cdots & F_{m-1}^\top\\
G_1^\top & G_0^\top & \cdots & G_2^\top & F_1^\top & F_0^\top & \cdots & F_2^\top
\end{bmatrix},
\]
with indices understood modulo $m$. If every $F_i$ commutes with every $G_j$,
then $\hat H_X\hat H_Z^\top=0$. If, moreover, the family
$F_0,\dots,F_{m-1}$ is pairwise commuting, then the left half of $\hat H_X$
contains a $2\times 3$ subarray of mutually commuting blocks, and hence the
Tanner graph contains a Fossorier-type closed block cycle of length $12$
~\cite{Fossorier2004}. The same conclusion holds if the family
$G_0,\dots,G_{m-1}$ is pairwise commuting. Therefore, to preserve the
Hagiwara--Imai orthogonality mechanism while excluding such Fossorier cycles,
one needs at least one noncommuting pair among the $F_i$'s and at least one
noncommuting pair among the $G_i$'s.
\end{proposition}

\begin{proof}
A direct block multiplication gives, for $j,k\in\{0,1\}$,
\[
(\hat H_X\hat H_Z^\top)_{j,k}
=
\sum_{\ell=0}^{m-1}
\bigl(
F_{\ell-j}G_{k-\ell}+G_{k-\ell}F_{\ell-j}
\bigr),
\]
where the indices are taken modulo $m$. If every $F_i$ commutes with every
$G_j$, then each summand is twice the same permutation matrix, hence vanishes
over $\mathbb{F}_2$. This proves the orthogonality.

Now assume that the family $F_0,\dots,F_{m-1}$ is pairwise commuting. Since
$m\ge 3$, the left half of $\hat H_X$ contains a $2\times 3$ subarray such as
\[
\begin{pmatrix}
F_0 & F_1 & F_2\\
F_{m-1} & F_0 & F_1
\end{pmatrix},
\]
and all six blocks in this subarray commute with one another. By
Fossorier's theorem on commuting $2\times 3$ permutation subarrays
~\cite{Fossorier2004}, this forces a closed block cycle of length $12$.
Exactly the same argument applies to the right half if the family
$G_0,\dots,G_{m-1}$ is pairwise commuting. Hence avoiding these
Fossorier-type cycles requires a noncommuting pair in each family.
\end{proof}

\bibliographystyle{IEEEtran}
\bibliography{ub_refs}

\end{document}